\documentclass[twoside]{article}

\usepackage{amsmath}
\usepackage{amsthm}
\usepackage{amsbsy}
\usepackage{bbm}
\usepackage{amssymb}
\usepackage{amsfonts}
\usepackage{cancel}
\usepackage{centernot}
\usepackage{cite}
\usepackage[dvips]{lscape}
\usepackage{graphicx}
\usepackage{tikz}
\usepackage{lipsum} 
\usepackage{hyperref}
\usepackage{verbatim}
\usepackage{titlesec}

\usepackage[margin=4.7cm]{geometry}

\titleformat*{\section}{\large\bfseries}
\titleformat*{\subsection}{\bfseries}

\theoremstyle{plain}
\newtheorem{theorem}{Theorem}[section]
\newtheorem{proposition}[theorem]{Proposition}
\newtheorem{lemma}[theorem]{Lemma}
\newtheorem{corollary}[theorem]{Corollary}
\newtheorem{remark}[theorem]{Remark}
\newtheorem{definition}[theorem]{Definition}

\newtheorem*{theorem*}{Theorem}
\let\d\delta
\let\eps\varepsilon
\let\n\noindent
\newcommand{\ov}{   \overline{\Pi}(\delta)   }
\newcommand{\var}{ \textup{Var}     }
\let\b\begin
\let\e\end
\let\f\frac
\let\bb\mathbb

\let\l\left
\let\r\right

\newpagestyle{main}{%
  \sethead[\thepage][][\thesection \quad \sectiontitle] % even
  {\hdrtitle}{}{\thepage}} % odd
\makeatother
\sethead[\normalsize{\thepage}][\small{ADAM BARKER}][] % even header
  {}{\small{FRACTAL-DIMENSIONAL PROPERTIES OF SUBORDINATORS}}{\normalsize{\thepage}} % odd header
\setfoot[][][] % even footer
  {}{}{} % odd footer

\date{}
\title{\normalsize{\textbf{FRACTAL-DIMENSIONAL~PROPERTIES~OF~SUBORDINATORS}}}% Covering the range of subordinators
\author{\normalsize{ADAM BARKER}}   
\begin{document}
\renewcommand{\baselinestretch}{1.5}

\maketitle

\paragraph{Abstract}  This work looks at the box-counting dimension of sets related to subordinators (non-decreasing L\'evy processes).  It was recently shown in \cite{s14} that  almost surely $\lim_{\d\rightarrow0}U(\d)N(t,\d) = t$,  where $N(t,\d)$ is the minimal number of boxes of size at most \( \delta\)  needed to cover a subordinator's range up to time $t$, and $U(\d)$ is the subordinator's renewal function. Our main result is a central limit theorem (CLT) for $N(t,\d)$, complementing and refining work in \cite{s14}. 

Box-counting dimension   is defined in terms of $N(t,\d)$, but for subordinators we prove that it can also be defined  using  a new   process  obtained by  shortening the original subordinator's  jumps of size greater than \(\delta\). This new process can be manipulated with remarkable ease in comparison to $N(t,\d)$, and allows better understanding of the box-counting dimension of a subordinator's range in terms of its L\'evy measure, improving upon  \cite[Corollary 1]{s14}.
Further, we shall prove corresponding CLT and almost sure convergence results for the new process.

\vspace{-0.25cm}

\section{Introduction \& Background} \vspace{-0.25cm}
We shall mostly study the minimal number, \(N(t,\delta)\), of intervals of length at most \( \delta\)  needed to cover the range $\{   X_s : 0\leq s \leq t\}$ of a subordinator $(X_s)_{s\geq0}$.  
The main result in this paper is a central limit theorem for \(N(t,\d)\), complementing the almost sure convergence result $ \lim_{\delta \rightarrow 0}U(\d)N(\delta,t)=t$, almost surely, where $U(\d)$ denotes the renewal function of the subordinator, see \cite[Theorem 1.1]{s14}.

Prior to the results in \cite{s14}, most works on box-counting dimension focused only on finding the value of $ \lim_{\delta \rightarrow 0} \log(N(t,\delta)) /  \log(1/\delta)    $, which defines the box-counting dimension. However, working with $N(t,\d)$ itself allows precise understanding of its fluctuations around its mean,  inaccessible at the log scale.

We will introduce an alternative ``box-counting scheme'' to \(N(t,\delta)\), which  allows us to   understand the dimension of the range in terms of the L\'evy measure, complementing results formulated in terms of the renewal function. 
%\newpage

The fractal dimensional study of sets such as the range or graph of L\'evy processes, and especially subordinators, has a very rich history. There are many works which study the  box-counting, Hausdorff,  and packing dimensions  of sets related to L\'evy processes  \cite{bg60,cps14,f04,fs98,ft92,h78,k09,ksx12,kx05,kx08,kx16,s14,sw82,x04}.

%In addition to the range, many works study the level sets or the graph of a L\'evy process. 
%Various measures of the dimension of such sets have been examined, such as box-counting, Hausdorff, and packing dimensions. 
 %In particular,  \cite[Chapters 4-5]{x04} gives an account of the study of fractal properties of the range and graph of Markov processes. 

% ,   denoted by     \( \underline{\text{dim}_{B}} \)  and   \( \overline{\text{dim}_{B}}\) . 

%Closely related to this is the upper (lower) modified box-counting dimension, defined as the infimum over all suprema of upper (lower) box-counting dimensions of members of countable covers of the set. This modification ensures the dimension of any countable set is zero, and allows us to consider non-compact sets. With the obvious notation, we have the following bounds for the Hausdorff dimension and packing dimension of any set (see \cite{f04}): \[   \text{dim}_{H} \leq   \underline{\text{dim}_{MB}}\leq   \underline{\text{dim}_{B}}, \qquad \text{dim}_{H} \leq \underline{\text{dim}_{MB}}  \leq\overline{\text{dim}_{MB}}  \leq \text{dim}_{P} \leq   \overline{\text{dim}_{B}}  .\]
%Moreover, for any subset of $\mathbb{R}^n$, $ \text{dim}_{P}$ and $\overline{\text{dim}_{MB}}$ are equal \cite{f04}.      

%\vspace{-1.4cm}
\n A L\'evy process is a stochastic process in $\bb{R}^d$ which has stationary, independent increments, and starts at the origin.  A subordinator $X:=(X_t)_{t\geq0}$ is a non-decreasing real-valued L\'evy process.  
The Laplace exponent $\Phi$ of a subordinator $X$ is defined by the relation $ e^{-\Phi(\lambda)} =   \mathbb{E}[  e^{- \lambda X_1}  ] $  for $\lambda\geq0$. By the   L\'evy Khintchine formula \cite[p72]{b98},   $\Phi$ can always be expressed as  
 \begin{equation}
 \label{lk}   \Phi(\lambda) = \ \text{d} \lambda +  \int_0^\infty  (1- e^{-\lambda x} ) \Pi(dx) , %\vspace{-0.1cm}
 \end{equation}
 where d is the linear drift, and $\Pi$ is the L\'evy measure, which determines the size and intensity of the jumps (discontinuities) of $X$, and satisfies the condition $\int_0^\infty (1\wedge x)\Pi(dx)<\infty$. 
The renewal function is the expected first passage time above $\d$,  \( U(\delta):=\bb{E}[T_\d]\),  where $T_\d:= \int_0^\infty \mathbbm{1}_{ \{  X_t \leq \delta  \} } dt $. 
%For more details, we refer to \cite[Chapter 3]{b98} or \cite[Chapter 2]{d07}.

If the L\'evy measure is infinite, then infinitesimally small jumps occur at an infinite rate, almost surely. We will not study processes with finite L\'evy measure, as they have only finitely many jumps, and hence no fractal structure.

The box-counting dimension of a set in $\bb{R}^d$ is  $  \lim_{\delta \rightarrow 0} \log(N(\delta)) /  \log(1/\delta)    $, where $N(\delta)$ is the minimal number of $d$-dimensional boxes of side length $\delta$ needed to cover the set. %Then roughly speaking, we have $N(\delta) \approx K \delta^{-\dim_B(A)}$ for some constant $K$. 
The limsup and liminf  respectively define  the upper and lower box-counting dimensions. For further background reading, we refer to \cite{b98,b99} for  subordinators,  \cite{d07,kyp06,s99} for L\'evy processes, and\cite{f04,x04} for fractals.
 
The paper is structured as follows: Section \mbox{\ref{mainresults}} outlines the statements of all of the main results; Section \mbox{\ref{ncltproof}}   contains the proof of the CLT result for $N(t,\d)$ and the lemmas required for this proof; Section \mbox{\ref{Lproofs}} contains the proofs of all of the main results on the new process $L(t,\d)$; Section \mbox{\ref{ext}}  extends this work to the graph of a subordinator, and considers the special case of a subordinator with regularly varying Laplace exponent.
 
\vspace{-0.37cm}
%\p{Results} Recall that $N(t,\delta)$ denotes the fewest intervals of length at most $\delta$ needed to cover the range  of a subordinator up to time $t > 0$.
\section{Main Results} \label{mainresults}\vspace{-0.32cm}
 \subsection{A Central Limit Theorem for $N(t,\delta)$ } 
\vspace{-0.3cm}
   Expanding upon Bertoin's result \cite[Theorem 5.1]{b99}, the following almost sure limiting behaviour of $N(t,\d)$ was determined by Savov  \cite[Theorem 1.1]{s14}. 
   \b{theorem}[Savov, 2014] \label{nslln} If a subordinator   has infinite L\'evy measure or a non-zero drift, then for all $t>0$, \(    \lim_{\d\rightarrow0+}   U(\delta) N(t,\delta) =   t \)    almost surely.
   \e{theorem} %\vspace{-0.1cm}
\n We will  complement and refine this work with a CLT on $N(t,\d)$. When the subordinator has no drift, we require a mild condition on the L\'evy measure:  
 \b{equation} \label{intcond}   \liminf_{\d\rightarrow0} \f{I(2\d)  }{ I(\d)  } >1,
 \e{equation}
where $I(u):=\int_0^u \overline{\Pi}(x)dx$, and $\overline{\Pi}(x):=\Pi((x,\infty))$. 
% \vspace{-0.1cm}
 \begin{remark} \textup{Condition $\l(\ref{intcond}\r)$  has many equivalent formulations, see   \cite[Ex.\ III.7]{b98} and \cite[Section 2.1]{bgt89}. We emphasise that $\l(\ref{intcond}\r)$ is far  less restrictive than regular variation (or even   $\mathcal{O}$-regular variation) of the Laplace exponent, and appears naturally in the context of   the law of the iterated logarithm (see e.g.\ \cite[p87]{b98}).}\end{remark}
% \vspace{-0.3cm} 
 \b{theorem} \label{nclt}  For every driftless subordinator with L\'evy measure satisfying $\l(\ref{intcond}\r)$,  for any $t>0$,  $N(t,\d)$ satisfies the following central limit theorem:
 \b{equation} \label{thm}     \f{  N(t,\delta)  - t a(\delta)  }{  t^\f{1}{2} b(\delta)  } \overset{d}\rightarrow  \mathcal{N}(0,1),
  \e{equation}  as $ \delta \rightarrow 0$, where $a(\delta) := U(\delta)^{-1}$, and $b(\delta) :=  U(\delta)^{-\f{3}{2}} \textup{Var}(T_\delta)^\f{1}{2} $.   \e{theorem}

%%%THE FOLLOWING IS ALL WORK ON L(t,\delta), WRITTEN UP LAST YEAR, STATEMENTS ONLY, PROOFS TO COME LATER ON

\subsection{An Alternative Box-Counting Scheme, $L(t,\delta)$ }

%In order to understand \ref{lclt} and \ref{lslln}, we need to use a new definition of box-counting dimension for subordinators. 
%When scaled by $\frac{1}{\delta}$, the ``process of $\delta$-shortened jumps''   %, using freedom under the relation ``$\asymp$'', see \ref{freedom}.

   \begin{definition} \label{shorteneddefn} \label{subord} The process of $\delta$-shortened jumps, $\tilde{X}^\delta:=(\tilde{X}_t^\delta)_{t\geq0}$, is obtained by shortening all jumps of $X$ of size larger than $\delta$  to instead have size  $\delta$.  That is,   $\tilde{X}^\delta$ is  the subordinator with Laplace exponent $\tilde{\Phi}^\d(u)= \textup{d}u + \int_0^\d (1-e^{-ux})\tilde{\Pi}^\d (dx)$ and  L\'evy measure   \(\tilde{\Pi}^\delta (dx) = \Pi(dx)\mathbbm{1}_{ \{  x<\delta  \}} + \overline{\Pi}(\delta)  \Delta_\delta \), where $ \Delta_\delta$ denotes a unit point mass at $\delta$, and $\Pi$ is the  L\'evy measure of $X$.     \end{definition}  
    % \begin{remark} \textup{ \label{subord} The process   $\tilde{X}^\delta$ is  a subordinator with Laplace exponent $\tilde{\Phi}^\d(u)= \textup{d}u + \int_0^\d (1-e^{-ux})\tilde{\Pi}^\d (dx)$, and  L\'evy measure   \(\tilde{\Pi}^\delta (dx) = \Pi(dx)\mathbbm{1}_{ \{  x<\delta  \}} + \overline{\Pi}(\delta)  \Delta_\delta \), where $ \Delta_\delta$ denotes a unit point mass at $\delta$, and $\Pi$ is the  L\'evy measure of $X$.}\end{remark}
   
         \begin{definition} For $\d,t>0$, $L(t,\d)$ is defined by   \(L(t,\delta) := \f{1}{\d} \tilde{X}_t^\delta    \).  \end{definition}

    \n  We will see in Theorem \mbox{\ref{asymp}} that $L(t,\d)$ can replace $N(t,\d)$ in the definition $  \lim_{\delta \rightarrow 0} \log(N(t,\delta)) /  \log(1/\delta)    $  of the box-counting dimension of the range of $X$. Then we will prove almost sure convergence and CLT  results for $L(t,\delta)$.% which  facilitate our proof of  Theorem \ref{nclt}, the main CLT result for $N(t,\delta)$. %   ,  we can extend the CLT from $L(t,\delta)$ to $N(t,\delta)$, using \ref{lslln}.  Now let us formally state the main results on \(L(t,\delta)\). 

      %We shall see that the scaled process, \( L(t,\delta)\), can be used to define the box-counting dimension of the range.

     \begin{remark} \textup{ \label{freedom} The log scale at which box-counting dimension is defined allows flexibility  among functions to be taken in place of the optimal count. In particular, there is freedom between functions related by $\label{freedom} f \asymp g $ asymptotically, where the notation means that there exist positive constants $A,B$ such that $A f(x) \leq g(x) \leq B f(x)$ for all $x$.     For more details, we refer to \cite[p42]{f04}. %We shall exploit this freedom with the definition $L(t,\delta):=\frac{1}{\delta}\tilde{X}_t^\delta$, see \ref{ldefn}, which aids many calculations when working with subordinators. 
}\end{remark}

\begin{theorem} \label{asymp} For all $\d,t>0$, for every subordinator,  $N(t,\delta) \asymp L(t,\delta) $. In particular, by Remark \mbox{\ref{freedom}},   $L(t,\delta)$ can be used to define the box-counting dimension of the range, i.e.\ $  \lim_{\delta \rightarrow 0} \log(N(t,\delta)) /  \log(1/\delta) =  \lim_{\delta \rightarrow 0} \log(L(t,\delta)) /  \log(1/\delta)    $. \end{theorem}

\begin{theorem} \label{lslln}  For every subordinator with  infinite L\'evy measure, for all $ t >0$,  \begin{equation}   \lim_{\d\rightarrow0} \frac{ L(t,\delta)}{ \mu(\delta)} =   t ,  \end{equation} almost surely, where
$\mu(\delta) :=  \frac{1}{\delta} ( \textup{d} + I(\delta) ) $, and $I(\d)= \int_0^\d \overline{\Pi}(y)dy $.
\end{theorem}

\begin{remark} \textup{ It can be deduced from \cite[Prop 1.4]{b99} that $U(\delta)^{-1} \asymp      \frac{1}{\delta} ( I(\delta) + \ \text{\textup{d}})$, for any subordinator.    Theorems \mbox{\ref{nslln}}, \mbox{\ref{asymp}} and \mbox{\ref{lslln}} allow us to understand this relationship in terms of geometric properties of subordinators.
}\end{remark}

\begin{theorem} \label{lclt} \hspace{-1.5mm}  For every subordinator with  infinite L\'evy measure, for all~$t>0$,   \begin{equation}   \frac{L(t , \delta) - t \mu(\delta)}{  \ t^{\f{1}{2}}   v(\delta) \f{}{} }    \overset{d}\rightarrow     \mathcal{N}(0,1)\end{equation} as $ \delta \rightarrow 0$, where  \( \mu(\delta) =  \frac{1}{\delta} ( \textup{d} + I(\delta) )  \), and \(  v(\delta):=   \f{1}{\d} \l[  \int_0^\infty  (x \wedge \d)^2 \Pi(dx)    \r]^{\frac{1}{2}}.\)
\end{theorem}

%$L(t,\delta)$ should be thought of as a helpful way to define $\text{dim}_B$, specific to subordinators. While it is by no means a canonical means of defining box-counting dimension, it has some very nice properties. 
\begin{remark} \label{meanvar}  \textup{   Applying Remark \mbox{\ref{subord}},  the L\'evy Khintchine formula $\l(\ref{lk}\r)$, and the fact that for any integrable function $f$,  \( \int_0^\d  f(x) \ \tilde{\Pi}^\delta (dx)  = \int_0^\infty  f(  x \wedge \delta)  \  \Pi(dx)\), it follows that  for all $\d,t>0$,  the mean and variance of $L(t,\d)$ are given by \[   \mathbb{E}[     L(t,\delta)] = t \mu(\d)   , \quad     \text{\textup{Var}}(   L(t,\delta) ) =   t   v(\delta) .\]  Computing the moments of $L(t,\d)$ is remarkably simple in comparison to the moments of $N(t,\d)$, which are not well known. This is a key benefit of  using $L(t,\d)$ to  study the box-counting dimension of the range  of a subordinator.   }\end{remark}

 \section{Proof of Theorem \mbox{\ref{nclt}}} \label{ncltproof}
 \subsection{A Sufficient Condition for Theorem \mbox{\ref{nclt}}} We will first work towards proving the following sufficient condition:
 \begin{lemma}  \label{suffcond}  For every subordinator with infinite L\'evy measure,  a sufficient condition for the convergence in distribution $\l(\ref{thm}\r)$, with $\sigma_\d^2:=\var(T_\d)$, is  \b{equation} \label{lemmm} \lim_{\d\rightarrow0}     \f{ U(\d)^\f{7}{3}     }{  \sigma_\d^2  } =0.
 \e{equation}
 \end{lemma}
 
\n The proof of  Lemma \mbox{\ref{suffcond}} relies upon the Berry-Esseen Theorem, a very useful result for proving central limit theorem results as it provides the speed of convergence, which is stated here in Lemma \mbox{\ref{be}}. See \cite[p542]{f08} for more details.    %There is no reason why $t \neq 1$ should be a problem - it just makes the computations clearer. 

\b{lemma}(Berry-Esseen Theorem) \label{be} Let $Z\sim \mathcal{N}(0,1)$. There exists a finite constant $c>0$ such that for every collection of iid random variables $(Y_k)_{k\in\bb{N}}$  with the same distribution as $Y$, where $Y$ has finite mean, finite absolute third moment, and finite non-zero variance,  for all $n\in \bb{N}$ and  $x \in \bb{R}$, 
\b{equation} \label{berry}   \l|   \bb{P}\l(     \f{  Y_1 - \bb{E}[Y] + \cdots + Y_n - \bb{E}[Y]   }{   \textup{Var}( Y  )^\f{1}{2}  \sqrt{n} }         \geq   x   \r)   - \bb{P}(   Z \geq x)      \r| \leq   \f{ c \bb{E}[|Y - \bb{E}[Y]|^3]    }{  \textup{Var}(Y)^\f{3}{2}   \sqrt{n}   }     .
\e{equation}

\e{lemma}

 \n For brevity, we will only provide calculations for $t=1$. The proofs for different values of $t$ are essentially the same. Recall the definitions $a(\d):= U(\d)^{-1}$,   $\sigma_\d^2:=\textup{Var}(T_\delta)$, and $b(\d):=  U(\delta)^{-\f{3}{2}}\sigma_\d  $. We shall aim to prove that for all $x\in\bb{R}$, 

  \[
   \lim_{\d\rightarrow0+} \l| \mathbb{P}\l(    \frac{ N(1,\delta) - a(\delta)}{b(\delta)} \leq x \r)  -  \mathbb{P}\l(    Z \leq x \r)     \r| =0.
 \] 

 \n For each $\d>0$, $\l(\ref{berry}\r)$ provides an upper bound, and  then under condition (\ref{intcond}), we can prove that this bound converges to zero as $\d\rightarrow0$.

 \b{proof}[Proof of Lemma \mbox{\ref{suffcond}}]  Let $T_{\delta}^{(k)}$ denote the $k$th time at which $N(t,\d)$ increases, and let $T_{\delta , k}$, $k \in \bb{N}$, denote iid copies of $T_\delta^{(1)}$. By the strong Markov property, $ T_{\delta}^{(k)}$ and $ \sum_{i=1}^kT_{\delta , i} $ have the same distribution. Then, with $n:= \lceil a(\delta) + x b(\delta) \rceil  $, where $\lceil \cdot \rceil$ denotes the ceiling function,
\begin{equation} \label{b} \mathbb{P}\l(    \frac{ N(1,\delta) - a(\delta)}{b(\delta)} \leq x \r) = \mathbb{P}\l(  N(1,\delta)  \leq a(\delta) + x b(\delta) \r),
\end{equation}  
\n and   since $N(1,\d)$ only takes integer values, using the fact that $T_\d^{(n)}$ has the same distribution as the sum of $n$ iid copies of $T_\d^{(1)}$, it follows that
\begin{equation}\begin{split} \l(\ref{b}\r)
    =   \mathbb{P}\big(  N(1,\delta)  \leq n  \big) &= \mathbb{P}\l(  T_\delta^{ (n)} \geq 1 \r)
   \hspace{0.2cm} =  \hspace{0.2cm}  \mathbb{P}\l(  \hspace{0.228cm} \sum_{i=1}^n T_{\delta , i}    \geq 1   \hspace{0.228cm}  \r)  
 \\
 &= \mathbb{P}\l(  \hspace{0.4cm}  \sum_{i=1}^n   \big(  T_{\delta , i}   -   U(\d)     \big)   \geq   1  -  nU(\d) \hspace{0.08cm}   \r)   %\begin{equation*}  =  \mathbb{P}\Bigg(  \frac{ T_{\delta,1}^{ 1 } - \mathbb{E}[   T_{\delta}^{ 1 } ] + \cdots +  T_{\delta,n_{x,\delta}}^{ 1} - \mathbb{E}[   T_{\delta}^{ 1 } ] }{  \sqrt{n_{\delta,x} \textup{Var}(T_\delta^1)  }   }  \leq  \frac{ 1  - n_{\delta,x} \mathbb{E}[   T_{\delta}^{ 1 } ]  }{  \sqrt{n_{\delta,x} \textup{Var}(T_\delta^1)  }   } \Bigg).
%\end{equation*}
\\ \label{probexpre}  &=  \mathbb{P}\l(  \frac{ \sum_{i=1}^n \big( T_{\delta , i}  -  U(\d) \big) }{   \sqrt{n \sigma_\d^2  }  }   \geq  \frac{ 1  - n  U(\d) }{  \sqrt{n \sigma_\d^2  }   } \r).  
\end{split} 
\end{equation}

\n It follows from Lemma  \mbox{\ref{momentm}} that  $\sigma_\d^2 \leq \bb{E}[T_\d^2]\leq c U(\d)^2$,  which then implies that $b(\d) =  o(a(\d))$ as $\d\rightarrow0$. Then, as $\d\to0$, the asymptotic behaviour of  $n$  is
 \[  n =  \lceil a(\delta) + x b(\delta) \rceil \sim  a(\delta)+ x b(\delta) = a(\delta) +o(a(\d)) \sim a(\d) = U(\d)^{-1}.\]  It follows, with $x^\prime$ depending on $x$ and $\d$, that as $\d\rightarrow0$, \vspace{0.1cm}
           \begin{align}
             \label{c} -x^\prime :&=   \frac{ 1  - n U(\d)  }{   \sqrt{n \sigma_\d^2  }    }     =      \frac{ 1  - \lceil  a(\delta) + x b(\delta)   \rceil U(\delta)  }{  \big(\lceil  a(\delta) + x b(\delta)   \rceil \big)^{\frac{1}{2}} \sigma_\d     }       \sim  \frac{ 1  - (  a(\delta) + x b(\delta)   ) U(\delta)  }{  (  a(\delta) + x b(\delta)   )^{\frac{1}{2}} \sigma_\d    } 
    \\  
          \label{xx}     &=   \frac{ 1  - 1 - x b(\delta)  U(\delta)  }{  ( a(\d) + x b(\delta)   )^{\frac{1}{2}}\sigma_\d    }     \sim    \hspace{0.26cm}   \frac{  - x b(\delta)  U(\delta)  }{   U(\delta)^{-\frac{1}{2}} \sigma_\d    } \hspace{0.26cm}  =      \frac{  - x b(\delta)  U(\delta)^\f{3}{2}  }{    \sigma_\d      } \hspace{0.16cm}  = \hspace{0.16cm} -x.  \end{align}

 \vspace{0.1cm}

   \n Now,  by the triangle inequality and symmetry of the normal distribution, combining $(\ref{probexpre})$ and $(\ref{xx})$,  it follows that as $\d\to0$, for any $x\in\bb{R}$, \vspace{0.1cm}
 \begin{equation} \begin{split} 
   \ \ \ \     \l| \mathbb{P}\l( \frac{N(1,\delta)-a(\delta)}{b(\delta)} \leq x\r) - \mathbb{P}\l(Z\leq x\r)\r| 
\leq \l| \mathbb{P}\l( Z \geq -x^\prime  \r) - \mathbb{P}\l(Z\geq -x\r)  \r|  &
\\
%\label{prob} 
  +  \l| \ \mathbb{P}  \l(  \frac{ 1}{   \sqrt{n \sigma_\d^2  }  }  \sum_{i=1}^n \l(T_{\delta , i}  -  U(\d) \r)  \geq  -x^\prime \r)   -  \  \mathbb{P}\l( \ Z\geq -x^\prime   \r) \   \r|   & 
\\
 \label{laststep}
      =  \l| \ \mathbb{P}  \l(  \frac{ 1}{   \sqrt{n \sigma_\d^2  }  }  \sum_{i=1}^n \l(T_{\delta , i}  -  U(\d) \r)  \geq  -x^\prime \r)   -  \  \mathbb{P}\l( \ Z\geq -x^\prime   \r) \   \r| &  + o(1). 
 \end{split}
 \end{equation}

 %Now, since $x^\prime \rightarrow x$ as $\delta\rightarrow0+$, it follows that $\l| \mathbb{P}\l( Z \geq x^\prime  \r) - \mathbb{P}\l(Z\geq -x\r)\r| \rightarrow0$.  So we need   (\ref{prob}) to converge to zero. 
 
 \n  Recall that we wish to show that $\l(\ref{laststep}\r)$ converges to zero. By  the Berry-Esseen Theorem and the fact that $n\sim U(\d)^{-1}$, it follows that as $\d\to0$, \vspace{0.1cm}
\b{equation*} 
  \l(\ref{laststep}\r) 
\leq 
 C      \f{     \bb{E}[|T_\delta-U(\delta)|^3]  }{  \sigma_\d^3  n^\f{1}{2}   }  + o(1)
 \sim 
 C \f{  U(\d)^\f{1}{2}   \bb{E}[|T_\delta-U(\delta)|^3]  }{   \sigma_\d^3      }     .
\e{equation*} 
 Applying the triangle inequality, then Lemma  \mbox{\ref{momentm}} with $m=2$ and $m=3$  to $\bb{E}[ |T_\d - U(\d)|^3  ]$, it follows that
\[ \l(\ref{laststep}\r)\leq 8C \f{U(\d)^\f{1}{2}  U(\d)^3}{   \sigma_\d^3  }    = 8C \l(\f{U(\d)^\f{7}{3}}{   \sigma_\d^2  }  \r)^{\f{2}{3}}.
\] 
Therefore if the condition $\l(\ref{lemmm}\r)$ as in the statement of Lemma \mbox{\ref{suffcond}} holds, then the desired convergence in distribution $\l(\ref{thm}\r)$ follows, as required.

% So if $\f{U(\d)^\f{7}{3}}{ \sigma_\d^2 } \rightarrow0$, then it follows that $\l(\ref{prob}\r)\rightarrow0$, as required.  
 \e{proof}

\newpage

\b{lemma}  \label{momentm}  For every subordinator with infinite L\'evy measure, for all $m\geq 1 $,
 \[
  \limsup_{\d\rightarrow0+} \f{\bb{E}[ T_\delta^m ]}{U(\delta)^m}  <\infty.  
 \]  
 \e{lemma}

  \b{proof}[Proof of Lemma \mbox{\ref{momentm}}]  First, by the moments and tails lemma (see \cite[p26]{k06}),
  \b{equation*}    
  \f{\bb{E}[T_\d^m]}{U(\d)^m} 
   =
     \bb{E}\l[\l(\f{T_\d}{U(\d)} \r)^m\r]  
      =\int_{0}^{\infty}m y^{m-1}\bb{P}\l(\f{T_\delta}{U(\delta)}>y\r)dy.
 \e{equation*}
 \n By the definition of $T_\d$, it follows that $ X_u \geq \d$ if and only if $T_\d \leq u $, and then
 \[  \f{\bb{E}[T_\d^m]}{U(\d)^m} 
 =\int_{0}^{\infty}m y^{m-1}\bb{P}(X_{yU(\delta)}\leq\delta)dy  = \int_{0}^{\infty}my^{m-1}\bb{P}(e^{ - \f{1}{\d} X_{yU(\delta)}}\geq e^{-1})dy.
 \]
   Now, applying Markov's inequality, the definition $\bb{E}[e^{-\lambda X_t}]=e^{-t\Phi(\lambda)}$, and  the fact that $U(\d)  \Phi(1/\d) \geq c$  for some constant $c$ (see \cite[Prop 1.4]{b99}),
		 \[ \f{\bb{E}[T_\d^m]}{U(\d)^m} 
		 \leq  
		         \int_{0}^{\infty} m y^{m-1}e^{1-yU(\delta)\Phi(1/\delta)}dy
		 \leq    \int_{0}^{\infty}my^{m-1}e^{1-cy}dy,
 \] which is finite and independent of $\d$. Therefore the $\limsup$ is finite, as required.
 \e{proof}

 %%%%NOW WE'RE DONE WITH THE PROOFS OF THE TWO LEMMAS

% \subsection{A Useful Consequence of the Integral Condition $\l(\ref{intcond}\r)$}

 %%%%%NOW WE START TOWARDS PROVING THE SUFFICIENT CONDITION HOLDS UNDER THE INTEGRAL CONDITION

%%%%%   SUFF COND  --->  JAIN AND PRUITT --->  tR(\lambda_t) <\infty SUFF   --->   \lambda_t \delta \asymp 1 SUFF   --->      \lambda_t \delta \asymp 1 HOLDS

 \subsection{Proof of Theorem \mbox{\ref{nclt}}}

 \n Theorem \mbox{\ref{nclt}} is proven by a contradiction, using Lemma \mbox{\ref{iic}} to show that the sufficient condition  in Lemma \mbox{\ref{ld}} holds.

%Lemma \ref{ld}, which is stated here and proven later, facilitates the proof of Theorem \ref{nclt}. Theorem \ref{nclt} will be proven by a contradiction, using Lemma \ref{iic}.
  
  \b{lemma} \label{iic} Recall the definition $I(\d):=\int_0^\d \overline{\Pi}(x)dx$. The condition $\l( \ref{intcond}   \r)$ implies that for each $\eta\in(0,1)$,  there exists a sufficiently large integer $n$ such that
  \b{equation} \label{improvedintcond}       \liminf_{\d\rightarrow0} \f{ I(\d)}{   I(2^{-n} \d)   } > \f{1}{\eta}.
  \e{equation}
 
  \e{lemma}

 \b{proof}[Proof of Lemma \mbox{\ref{iic}}]
\n The integral condition $\l( \ref{intcond}   \r)$ imposes that for some $B>1$,
\begin{equation} \label{1}  \liminf_{\delta \rightarrow0}  \f{ I(\delta)}{I(\delta/2)}   =  \liminf_{\delta \rightarrow0}  \f{  \int_0^\delta  \overline{\Pi}(y)(dy)     }{  \int_0^{\delta/2} \overline{ \Pi}(y) dy    }  =B.
\end{equation}
Then, by effectively replacing $1/2$ with $2^{-n}$ (so $1/2$ is replaced by a smaller constant), we can replace $B$ with $B^n$, which can be made arbitrarily large by choice of $n$. This follows by splitting up the fraction,
 \begin{align*}   & \liminf_{\delta \rightarrow0+}  \f{ I(\delta)}{I(2^{-n}\delta)}   =   \liminf_{\delta \rightarrow0+} \l( \f{ I(\delta)}{I(2^{-1}\delta)}  \f{ I(2^{-1}\delta)}{I(2^{-2}\delta)}   \cdots \f{ I(2^{-(n-1)}\delta)}{I(2^{-n}\delta)}  \r)   
\\
 \geq    & \liminf_{\delta \rightarrow0+} \l(  \f{ I(\delta)}{I(2^{-1}\delta)}\r)   \liminf_{\delta \rightarrow0+} \l( \f{ I(2^{-1}\delta)}{I(2^{-2}\delta)}  \r) \cdots \liminf_{\delta \rightarrow0+} \l( \f{ I(2^{-(n-1)}\delta)}{I(2^{-n}\delta)}  \r)  = B^n    > \f{1}{\eta},
\end{align*}

\n where we simply take $n$ sufficiently large that $B^n > 1/\eta$. 

\e{proof} 
 \n Using Lemma \mbox{\ref{iic}} for a contradiction is the step in the proof of Theorem \mbox{\ref{nclt}} which requires the condition $\l(\ref{intcond} \r)$. In order to prove Theorem \mbox{\ref{nclt}}, we require the  notation introduced in Definition \mbox{\ref{gR}}. We refer to \cite[p93]{jp87} for more details.  
 \b{definition} 
 \label{gR} Recalling from Remark \mbox{\ref{subord}} that the process $\tilde{X}^\d$ has Laplace exponent  $\tilde{\Phi}^{\delta}(u)= \textup{d} u + \int_0^\d (1-e^{-ux})\Pi(dx) + (1-e^{-u\d}) \ov$, we define:

\n  (i) $g(u):= \f{d}{du}\tilde{\Phi}^{\delta}(u) = \textup{d} + \int_0^\d x e^{-ux} \tilde{\Pi}^\d (dx)  $,

\n (ii) $ R(u):= \tilde{\Phi}^{\delta}(u) - ug(u)   = \int_0^\d    \l(1 - e^{-ux}(1+ux) \r)\tilde{\Pi}^\d(dx)$,

\n (iii)  $\lambda_\d$  denotes the unique solution to  $g(\lambda_\d) = x_\d$, for $\textup{d}<x_\d< \textup{d} + \int_0^\d x \tilde{\Pi}^\d(dx)  $.

 \e{definition}
 
\n One can ignore the drift  $\textup{d}$ in Definition \mbox{\ref{gR}}, since $\textup{d}=0$ throughout Section \mbox{\ref{ncltproof}}. The proof of Theorem \mbox{\ref{nclt}} now requires the following lemma:
 
   \b{lemma} \label{ld} For $\alpha>0$, $t= (1+\alpha)U(\d) $,  and $g(\lambda_\d) = x_\d= \d/t$, if
  \[\limsup_{\d\rightarrow0} \ \d  \lambda_\d < \infty,
  \]
   then the desired convergence in distribution   $\l(\ref{thm}\r)$,  as in Theorem \mbox{\ref{nclt}}, holds.
   \e{lemma}
 
 \b{proof}[Proof of Theorem \mbox{\ref{nclt}}]

\n Assume for a contradiction that there exists a sequence $(\d_m)_{m\geq1}$ converging to zero, such that $\lim_{m\to
\infty} \lambda_{\d_m} \d_m = \infty$. That is to say, assume that the sufficient condition in Lemma \mbox{\ref{ld}} doesn't hold.  For brevity, we omit the dependence of $\d_m$ on $m$.  Hence for all fixed $\eta,n>0$,  $ \eta\geq e^{-\lambda_\d 2^{-n} \delta} $ for all small  enough $\d>0$. By Fubini's theorem,  $I(\d)  \hspace{-0.05cm} =   \hspace{-0.05cm} \int_0^\d \overline{\Pi}(x)dx   \hspace{-0.05cm}  =   \hspace{-0.05cm}  \int_0^\d x \tilde{\Pi}^\delta (dx) $,  so
\b{align}  
\nonumber
\eta I(\delta)   \hspace{-0.05cm}  +    \hspace{-0.05cm}  I( 2^{-n}\d) &\geq e^{-\lambda_\d 2^{-n} \d}I(\delta)   \hspace{-0.05cm}   +   \hspace{-0.05cm}  I( 2^{-n} \d) 
  \geq e^{-\lambda_\d  2^{-n}\d}     \hspace{-0.05cm}   \int_0^{\delta } x \tilde{\Pi}^\delta (dx) +   \hspace{-0.05cm}  \int_0^{ 2^{-n}\d }      \hspace{-0.3cm}  x \Pi (dx)
\\
  \label{ttrr}  &= e^{-\lambda_\d  2^{-n}\d}\delta \overline{\Pi}(\delta) +  e^{-\lambda_\d  2^{-n}\d}\int_0^{\delta } x \Pi (dx) +     \hspace{-0.05cm}  \int_0^{ 2^{-n}\d }   \hspace{-0.5cm}  x \Pi (dx).  
\end{align} 
Removing part of the first integral and noting $1 \geq e^{-\lambda_\d x}$ for all $x>0$,
\[
 \l( \ref{ttrr} \r)  \geq e^{-\lambda_\d  2^{-n} \d}\delta \overline{\Pi}(\delta) +  \int_{ 2^{-n}\d}^{\delta } e^{-\lambda_\d 2^{-n} \d} x \Pi (dx) + \int_0^{ 2^{-n}\d } e^{-\lambda_\d x} x \Pi (dx).
\] 
Now,  $ e^{-\lambda_\d 2^{-n} \d} \geq e^{-\lambda_\d x}$  for $x\geq 2^{-n}\d$.  So for  $g(\lambda_\d)=x_\d = \f{\d}{ (1+\alpha)U(\d)  }$, where $\alpha>0$ is fixed and chosen sufficiently large that $x_\d<\int_0^\d x \tilde{\Pi}^\d(dx)$ for all $\d$ (this is possible  by the relation $U(\d)^{-1}\asymp I(\d)/\d$, see \cite[Prop 1.4]{b99}), 
\b{align*}
  \l( \ref{ttrr} \r)    &\geq e^{-\lambda_\d  2^{-n}\d}\delta \overline{\Pi}(\delta) +  \int_{ 2^{-n}\d}^{\delta } e^{-\lambda_\d x} x \Pi (dx) + \int_0^{ 2^{-n}\d } e^{-\lambda_\d x} x \Pi (dx) \hspace{0.52cm}
\\
 &= e^{-\lambda_\d  2^{-n}\d}\delta \overline{\Pi}(\delta) +  \int_0^{\delta } e^{-\lambda_\d x} x \Pi (dx)
  \geq g(\lambda_\d)  = \frac{\delta}{(1+\alpha)U(\delta)}
\geq  \frac{I(\delta)}{(1+\alpha)K} ,
\end{align*}

\n where the last two inequalities respectively follow by Definition \mbox{\ref{shorteneddefn}}, Definition \mbox{\ref{gR}} $(i)$ with $\textup{d}=0$, and the relation $U(\d)^{-1} \asymp I(\d)/\d$, see \cite[p74]{b98}.  So   for a  constant $K>0$, for    all sufficiently small $\delta>0$, we have shown $\eta I(\delta) + I( 2^{-n}\d) \geq \frac{I(\delta)}{(1+\alpha)K}$.  

 \n  Taking $\eta>0$ small enough that $ \frac{1}{(1+\alpha)K} \geq 2\eta  $, it follows that $    I(2^{-n} \delta)  \geq \eta I(\delta)$, and hence $ \label{eta}   I(\delta)/I(2^{-n} \delta) \leq  1/\eta$.
%
%\n But Lemma \ref{iic} tells us that for every $\eta>0$, $I(\d) / I(\d/2) > 1/\eta$, a contradiction.
%
 But in Lemma \mbox{\ref{iic}} we showed that for each fixed $\eta>0$, there is sufficiently large $n$ such that $ \liminf_{\d\to0} I(\d) / I(2^{-n}\d) > 1/\eta$, which is  a contradiction,  so the sufficient condition as in Lemma \mbox{\ref{ld}} must hold.

\e{proof}

%\n Now, taking $\beta:= \frac{1}{\alpha}$, we have $\frac{ h(\beta \delta)}{h(\delta)}  \leq \frac{1}{\eta}$ for all sufficiently small $\delta$. %Then if we take $\frac{1}{\eta} \in (1,B_\beta)$, % i.e.\  take $\eta$ sufficiently close to $1$,

 \b{remark} \textup{For a driftless subordinator, Theorem \mbox{\ref{nclt}} holds under the same condition $\l(\ref{intcond}\r)$  applied to the function $H(y):= \int_0^y x \Pi(dx)$  rather than the integrated tail function $I$. The integrated tail $I(y)=H(y) + y\overline{\Pi}(y)$ depends on  the large jumps of $X$ since $\overline{\Pi}(x)=\Pi((x,\infty))$, but  $H$ does not depend on the large jumps, so these  conditions are substantially different. 
 \\
 With only minor changes,   the argument  as in the proof of Theorem \mbox{\ref{nclt}} works with $H$ in place of $I$. Under condition $\l(\ref{intcond}\r)$ for $H$ in place of $I$, one can prove that Lemma \mbox{\ref{iic}} holds with $H$ in place of $I$. Then    we assume for a contradiction that there exists a sequence $(\d_m)_{m\geq1}$ converging to zero, such that $\lim_{m\to
\infty} \lambda_{\d_m} \d_m = \infty$. But  then as in the proof of Theorem \mbox{\ref{nclt}}, one can deduce that $\eta H(\d) + H(2^{-n}\d) \geq \f{1}{(1+\alpha)K^\prime}H(\d) $, which  contradicts the analogous Lemma \mbox{\ref{iic}} result with $H$ in place of $I$.  }
 \e{remark}

  \b{remark} \textup{Theorem \mbox{\ref{nclt}} can also be proven for subordinators with a  drift $\textup{d}>0$, under a stronger regularity condition. For $Y_t:=X_t - \textup{d}t$,  define $\Phi_Y$ as the Laplace exponent of $Y$. The convergence in distribution $\l( \ref{thm}\r)$ holds whenever $\limsup_{x\to0} x^{-5/6} \Phi_Y(x) <\infty$. This is proven using Remark \mbox{\ref{driftt}}, the inequality $\bb{P}(Y_t< a) \geq  1 - C t h( a)$ for all L\'evy processes (see \cite[p954]{p81} for details), and the asymptotic expansion of $U(\d)$ as in \cite[Theorem 4]{ds11}.}
 \e{remark}

  \subsection{Proofs of Lemmas \mbox{\ref{suffcondtoliminf}}, \mbox{\ref{trlt}}, \mbox{\ref{ld}}}

  Lemmas \mbox{\ref{suffcondtoliminf}}, \mbox{\ref{trlt}}, and \mbox{\ref{ld}} give  sufficient conditions for Theorem \mbox{\ref{nclt}} to hold.  The proofs for these lemmas are facilitated by Lemma \mbox{\ref{jp}}, which was proven in 1987 by Jain and Pruitt  \cite[p94]{jp87}. Recall that $\tilde{X}^\d $ denotes the process with $\d$-shortened jumps, as defined in Definition \mbox{\ref{shorteneddefn}}.

 \b{lemma} \label{suffcondtoliminf}   The convergence in distribution   $\l(\ref{thm}\r)$ as in Theorem \ref{nclt} holds if for some $\alpha\in(0,1]$,  \(  \liminf_{\delta\rightarrow0} \l[ \bb{P}\l(   \tilde{X}^\d_{(1+\alpha)U(\delta) } \leq  \delta    \r)  +  \bb{P}\l(     \tilde{X}^\d_{(1-\alpha)U(\delta) } \geq  \delta    \r)  \r] >0\). 

 \e{lemma}

 \b{proof}[Proof of Lemma \ref{suffcondtoliminf}]
 For all $\alpha>0$, recalling that $\bb{E}[T_\d] = U(\d)$,
 \[
 \sigma_\d^2 = \var(T_\d) \geq \var(T_\d ; |T_\d - U(\d)| \geq \alpha U(\d)) 
  \] 
 \[ 
 \geq \alpha^2 U(\d)^2 [ \bb{P}(  T_\d \geq (1+\alpha)U(\d)) +  \bb{P}(  T_\d \leq (1-\alpha)U(\d) )].
 \] For the desired convergence in distribution $\l(\ref{thm}\r)$ to hold, it is sufficient by Lemma \mbox{\ref{suffcond}} to show that $ \lim_{\d\rightarrow0}  U(\d)^\f{7}{3} / \sigma_\d^2=0$. Now,
 \[
      \f{U(\d)^\f{7}{3}}{\sigma_\d^2} \leq  \f{U(\d)^\f{1}{3}}{  \alpha^2 [  \bb{P}(  T_\d \geq (1+\alpha)U(\d) )  + \bb{P}(  T_\d \leq (1-\alpha)U(\d) )   ]  }.
 \]
Note that $T_\d \geq t  $ if and only if $\tilde{X}^\d_t \leq \d $ since jumps of size larger than $\d$ do not occur in either case, and so $X_t=\tilde{X}^\d_t$ when $T_\d\geq t$. It follows that $\l(\ref{thm}\r)$ holds if 
 \[  \liminf_{\delta\rightarrow0} \l[ \bb{P}\l(   \tilde{X}^\d_{(1+\alpha)U(\delta) } \leq  \delta    \r)  +  \bb{P}\l(   \tilde{X}^\d_{(1-\alpha)U(\delta) } \geq  \delta    \r)  \r]      >0. 
 \]
 \e{proof}
   \b{remark}\label{driftt}
The condition in Lemma \mbox{\ref{suffcondtoliminf}} is not optimal. If for $\eps\in\l(0,\f{1}{6}\r)$,  \( \lim_{\delta\rightarrow0} U(\d)^{2\eps - \f{1}{3}} \l[ \bb{P}\l(   \tilde{X}^\d_{U(\delta) + U(\delta)^{1+\eps} } \leq  \delta    \r) + \bb{P}\l(   \tilde{X}^\d_{U(\delta) - U(\delta)^{1+\eps} } \geq  \delta    \r) \r]  =\infty,\) then the convergence in distribution $\l(\ref{thm}\r)$ follows too. This stronger condition does not lead to any more generality than the condition $\l(\ref{intcond}\r)$ for driftless subordinators.
 \e{remark}

  \b{lemma} [Jain, Pruitt \hspace{-6pt} \mbox{\cite[Lemma 5.2]{jp87}}]  \label{jp}   There exists $c>0$ such that for every $\eps>0$,  $t\geq0$ and $x_\d>0$ satisfying $\textup{d} = g(\infty) < x_\d < g(0) = \textup{d} + \int_0^\d x \tilde{\Pi}^\d(dx)  $,
 \b{equation} \label{ii} \bb{P}\l(  \tilde{X}^\d_t \leq t x_\d\r)    \geq     \l(  1 - \f{  (1+\eps)c }{\eps^2 tR(\lambda_\d)    }     \r)e^{ - (1+2\eps)tR(\lambda_\d)}.
 \e{equation}
 
 \e{lemma}

 %\b{remark}  Jain and Pruitt's Lemma is key  to finding a lower bound on the probability in Lemma \ref{suffcondtoliminf}, which then leads via Lemma \ref{trlt} and  Lemma \ref{ld} to our proof of Theorem \ref{nclt}. \e{remark}

 \b{lemma} \label{trlt}    
 For $\alpha>0$, $t= (1+\alpha)U(\d) $,  and $g(\lambda_\d) = x_\d= \d/t$, if
 \[ \limsup_{\delta\rightarrow0} \  tR(\lambda_\d) < \infty ,\]
   then the desired convergence in distribution   $\l(\ref{thm}\r)$,  as in Theorem \mbox{\ref{nclt}}, holds.
 
 \e{lemma}

  \b{proof}[Proof of Lemma \mbox{\ref{trlt}}]
 
 \n Applying the inequality (\ref{ii}) from Lemma \mbox{\ref{jp}},  
 \b{equation} \label{jpapply} \bb{P}\l(   \tilde{X}^\d_{(1+\alpha)U(\delta) } \leq  \delta    \r)   \geq     \l(  1 - \f{  (1+\eps)c }{\eps^2 tR(\lambda_\d)    }     \r)e^{ - (1+2\eps)tR(\lambda_\d)}. 
 \e{equation}
 
 \n Now, letting $\limsup_{\d\rightarrow0} tR(\lambda_\d)<\infty$, we will consider two separate cases:
 \newline (i) If $\liminf_{\d\rightarrow0}tR(\lambda_\d)=\beta >0$, then by choice of $\eps>0$ such that $\f{1+\eps}{\eps^2} = \f{\beta}{2c}$, the lower bound in (\ref{jpapply}) is larger than a positive constant as $\delta\rightarrow0$. 
 
 \n(ii) If $\liminf_{\delta\rightarrow0} tR(\lambda_\d) = 0$, then imposing $\eps =  2c/(tR(\lambda_\d))$, the lower bound in (\ref{jpapply}) is again larger than a positive constant as $\delta\rightarrow0$. The desired convergence in distribution  $\l(\ref{thm}\r)$  then follows in each case by Lemma \mbox{\ref{suffcondtoliminf}}.   %Then by Lemma \ref{d=0} (which we prove below), the CLT result is proven whenever $(\ref{intcond})$ holds and there is no drift.  

 \e{proof}

%  \b{lemma} \label{ld} For $\alpha>0$, $t= (1+\alpha)U(\d) $,  and $g(\lambda_\d) = x_\d= \d/t$, if
 % \[\limsup_{\d\rightarrow0} \ \d  \lambda_\d < \infty,
 % \]
  % then the desired convergence in distribution   $\l(\ref{thm}\r)$,  as in Theorem \mbox{\ref{nclt}}, holds.
  % \e{lemma}

\b{proof}[Proof of Lemma \mbox{\ref{ld}}]
Noting that $1 - e^{-y}(1+y) \leq y$ for all $y>0$, 
\[tR(\lambda_\d) =  (1+\alpha)U(\d) \int_0^\d    ( 1 - e^{-\lambda_\d x}(1+\lambda_\d x) )\tilde{\Pi}^\d(dx)
\]
\b{equation} \label{improvedlemma}
 \leq  (1+\alpha)U(\d) \int_0^\d \lambda_\d x \tilde{\Pi}^\d(dx)   =  (1+\alpha)  U(\d)  \l( \int_0^\d x \Pi(dx) + \d\ov \r) \lambda_\d .
\e{equation}
 Then by the relation $U(\d) I(\d) \leq C \d$ for a constant $C$ (see \cite[Prop 1.4]{b99}),
\[ \l(\ref{improvedlemma}\r) = (1+\alpha) U(\d) I(\d) \lambda_\d \leq C     \d \lambda_\d.
\] So we can conclude that if $\limsup_{\d\rightarrow0}    \d \lambda_\d < \infty$, then the desired convergence in distribution   $\l(\ref{thm}\r)$ follows by Lemma \mbox{\ref{trlt}}.  

\e{proof}

 %%%%The following is redundant
  % So now let's prove that $\delta \asymp 1/\lambda_t$.  To see that $\lambda_t \delta$ cannot converge to zero, observe that for some $K>0$, again using \cite[Proposition III.1]{b98},
%\[   e^{-\lambda_t \delta}h(\delta)\leq g(\lambda_t)=\frac{\delta}{(1+a)U(\delta)}\leq  \frac{K}{1+a}h(\delta):= \frac{K}{1+a} \int_0^\delta \overline{\Pi}(y)dy .
%\]
%Then, taking $a$ sufficiently large, we see that $e^{-\lambda_t \delta} \leq \frac{1}{2} $, and hence $\lambda_t \delta$ cannot converge to zero.

%%%%NOW WE'RE UP TO THE PROOFS FOR L - WE'RE COMPLETELY DONE WITH N NOW OTHER THAN MAYBE SOME CLOSING REMARKS
\vspace{-0.56cm}

\section{Proofs of Results on $L(t,\d)$} \label{Lproofs}

\vspace{-0.36cm}

Firstly, we prove Theorem \mbox{\ref{asymp}}, which confirms that $L(t,\delta)$ can replace $N(t,\d)$ in the definition of the box-counting dimension of the range. This is done by showing that $L(t,\d)\asymp N(t,\d)$, which is known to be sufficient by Remark \mbox{\ref{freedom}}.

\begin{proof}[Proof of Theorem \mbox{\ref{asymp}}]  The jumps of the original subordinator $X$  and the process with shortened jumps $\tilde{X}^\delta$ are all the same size, other than  jumps bigger than size $  \delta$. The optimal number of intervals to cover the range,  $N(X,t,\delta)$,  always increases by $1$ at each jump bigger than size $\d$, regardless of its size, so it follows that $N(X,t,\delta)=N(\tilde{X}^\d,t,\delta)$, with the obvious notation. 

\n Instead of counting the number $N(X,t,\d)$  of boxes  needed to cover the range of $X$, consider those needed for the range of the subordinator $X^{(0,\delta)}$ with L\'evy measure $\Pi(dx) \mathbbm{1}_{\{x<\delta\}}$ (so all jumps of size larger than $\d$ are removed), and adding $Y_t^\d$, which counts the number of jumps  larger than size $\d$ of $X$.  It follows that $N(X,t,\d) \leq N(X^{(0,\delta)},t,\d) + Y_t^\d     \leq 2N(X,t,\d)$.

\n Consider  $M(X^{(0,\delta)},t,\delta)$, the number of intervals in a lattice of side length \(\delta\) to intersect with the range of  $X^{(0,\delta)}$. It is easy to show that $N(t,\d) \asymp M(t,\delta)$ (see \cite[p42]{f04}). Also, $M(X^{(0,\delta)},t,\delta)= \lceil  \f{1}{\d} X_t^{(0,\delta)}  \rceil$, since $X^{(0,\delta)}$ has no jumps of size larger than $\d$.  Now,  $ \f{1}{\d} X_t^{(0,\delta)} \asymp \lceil \f{1}{\d}  X_t^{(0,\delta)}  \rceil$ for  small enough $\d$, and hence
\[ 
 L(X,t,\delta) =  \frac{1}{\delta}  \tilde{X}_t^\delta = \frac{1 }{\delta} X_t^{(0,\delta)}+ Y_t^\d  \asymp M(X^{(0,\delta)},t,\delta) + Y_t^\d\]
 \[ \asymp N(X^{(0,\delta)},t,\delta) + Y_t^\d \asymp N(X,t,\d) .
 \]
By Remark \mbox{\ref{freedom}},  $\lim_{\d\rightarrow0}  \f{ \log(L(t,\d)) }{\log(1/\d)  } =  \lim_{\d\rightarrow0} \f{ \log(N(t,\d)) }{\log(1/\d)  }$, and hence $L(t,\d)$ can be used to define the box-counting dimension of the range of any subordinator.

  \qedhere
  \end{proof}

\n Next we will prove the CLT result for $L(t,\delta)$, working with $t=1$ for brevity. The proof is essentially the same for other  values of $t>0$. We will show convergence of the Laplace transform of $\f{1}{v(\d)} (L(1,\d) - \mu(\d) )  $ to that of the standard normal distribution. Recall that $Z\sim \mathcal{N}(0,1)$ has Laplace transform  $\mathbb{E}[ e^{- \lambda Z} ] = e^{\lambda^2 /2}$.

\begin{proof}[Proof of Theorem \mbox{\ref{lclt}}]  

 By Remark \mbox{\ref{subord}} and $\l(\ref{lk}\r)$, $ \delta L(t,\delta) = \tilde{X}_t^\delta$ is a subordinator with Laplace exponent $ \tilde{\Phi}^\delta$, and it follows that for any $\lambda\geq0$,
       \b{equation*}
    \lim_{\d\rightarrow0}   \mathbb{E}\l[ \textup{exp}\l( -\lambda   \frac{L(1,\d) - \mu(\delta) }{ v(\delta) } \r) \r] \hspace{-0.03cm}   = \hspace{-0.03cm}  e^{\f{\lambda^2}{ 2}} \hspace{-0.14cm}   \iff \hspace{-0.14cm}   \lim_{\d\rightarrow0} \l(  \frac{\lambda  \mu(\delta)  }{v(\delta)}   - \tilde{\Phi}^\delta \l(\frac{\lambda}{\d v(\delta)}\r)  \r)\hspace{-0.03cm} = \hspace{-0.03cm} \f{\lambda^2}{2}  . 
        \e{equation*}  %Taking the Taylor expansion, we can express the Laplace exponent as
      %  \[  
        %  \tilde{\Phi}^\delta \Big(\frac{\lambda}{g(\delta)}\Big)   =  \frac{\textup{d} \lambda}{  g(\delta)  } +  \int_0^\infty  \Big(\frac{\lambda x}{g(\delta)}      - \frac{\lambda^2 x^2}{2! g(\delta)^2}  + \frac{\lambda^3 x^3}{3! g(\delta)^3} - \dots \Big) \    \tilde{\Pi}^\delta (dx). 
       %  \]
     Recalling  the definition $\mu(\d)=\f{1}{\d}(\textup{d}+I(\d))$, where $I(\d):= \int_0^\d x \tilde{\Pi}^\d (dx)$, and writing $\tilde{\Phi}^\d$ in the L\'evy Khintchine representation as in  $\l(\ref{lk} \r)$,  it follows that
 \[
        \frac{\lambda  \mu(\delta)  }{v(\delta)}   - \tilde{\Phi}^\delta \l(\frac{\lambda}{\d v(\delta)}\r)  =      \frac{\lambda   ( \textup{d} + I(\delta) ) }{ \d v(\d)   }  -  \f{ \textup{d} \lambda }{  \d v(\d)   }   - \int_0^\d (1- e^{- \frac{\lambda x}{\d v(\delta)}}  ) \tilde{\Pi}^\delta (dx)
       \]
       \b{equation}\label{terms}
        =    \frac{\lambda  I(\delta) }{ \d v(\d)   }     - \hspace{-0.05cm}\int_0^\d \hspace{-0.05cm} (1- e^{- \frac{\lambda x}{\d v(\delta)}}  ) \tilde{\Pi}^\delta (dx) =  \hspace{-0.05cm}  \int_0^\d \frac{\lambda x}{ \d v(\d)   } \tilde{\Pi}^\d(dx)     - \hspace{-0.05cm}\int_0^\d \hspace{-0.05cm} (1- e^{- \frac{\lambda x}{\d v(\delta)}}  ) \tilde{\Pi}^\delta (dx).
       \e{equation}
Then applying the fact that $  \f{y^2}{2} -  \f{y^3}{6} \leq  y  - 1  + e^{-y} \leq  \f{y^2}{2} $ for all $y>0$, 
\[ 
   \int_0^\d \l( \frac{\lambda^2 x^2}{2 \d^2 v(\d)^2   } - \frac{ \lambda^3 x^3}{ 6 \d^3 v(\delta)^3}  \r) \tilde{\Pi}^\d(dx)  
    \leq   \l(\ref{terms}\r)
     \leq  \int_0^\d \frac{\lambda^2 x^2}{ 2 \d^2 v(\d)^2   } \tilde{\Pi}^\d(dx)      .
  \]  
  %\b{equation*} \label{lapl}
  %\tilde{\Phi}^\delta (u)   = \textup{d} u   + \int_0^\infty (1- e^{- u x}  ) \tilde{\Pi}^\delta (dx),
   %\e{equation*}
By the definition of $v(\d)$, it follows that $v(\d)^2 = \f{1}{\d^2} \int_0^\d x^2 \tilde{\Pi}^\d (dx)$, and so 
\[
\int_0^\d  \frac{ \lambda^2 x^2}{2 \d^2 v(\delta)^2}  \tilde{\Pi}^\d(dx)  = \f{\lambda^2}{2} .
\] 
It is then sufficient, in order  to show that $\l(\ref{terms}\r)$ converges to $\f{\lambda^2}{2}$, to prove that
\b{equation} \label{lcltcond}
\lim_{\d\rightarrow0}   \int_0^\infty  \frac{ x^3}{\d^3 v(\delta)^3}   \  \tilde{\Pi}^\delta(dx)   = 0.
 \e{equation}
%We need only consider the third order  terms, because convergence of the higher order terms follows.
Again by the definition of $v(\d)$, for $\l(\ref{lcltcond}\r)$ to hold we require both
 \begin{align} \label{int1}  \lim_{\d\rightarrow0} \frac{    \int_0^\delta x^3 \Pi(dx)    }{    (  \int_0^\delta x^2 \Pi(dx)   + \delta^2 \overline{\Pi}(\delta) )^{\f{3}{2}}  }  = 0,
 \\
  \label{int2}  \lim_{\d\rightarrow0}  \frac{    \d^3  \overline{\Pi}(\delta) }{  (  \int_0^\delta x^2 \Pi(dx)   + \delta^2 \overline{\Pi}(\delta) )^{\f{3}{2}}  } =0.
   \end{align}
\n Squaring the expression in $\l(\ref{int1}\r)$, since $x\leq\d$ within each integral, it follows that
 \b{equation*}  \frac{   \big( \int_0^\delta x^3 \Pi(dx)\big)^2    }{   \big(  \int_0^\delta x^2 \Pi(dx)   + \delta^2 \overline{\Pi}(\delta) \big)^{3}  }  \leq   \frac{   \delta^2 \big( \int_0^\delta x^2 \Pi(dx)\big)^2    }{   \big(  \int_0^\delta x^2 \Pi(dx)   + \delta^2 \overline{\Pi}(\delta) \big)^{3} } .
\e{equation*}
 By the binomial expansion, $(a+b)^3 \geq 3a^2b$ for  $a,b>0$,  and then
   \[ 
   \l( \ref{int1}  \r)  \leq   \frac{    \delta^2 \big( \int_0^\delta x^2 \Pi(dx)\big)^2    }{  3 \big(  \int_0^\delta x^2 \Pi(dx)  \big)^2 \big(  \delta^2 \overline{\Pi}(\delta) \big) } 
      = \frac{1}{3 \overline{\Pi}(\delta)} \rightarrow 0    ,
 \] 
 since the L\'evy measure is infinite. For $\l(\ref{int2}\r)$, simply observe that 
 \[ 
 \frac{     \d^3 \overline{\Pi}(\delta) }{   (  \int_0^\delta x^2 \Pi(dx)   + \delta^2 \overline{\Pi}(\delta) )^{\f{3}{2}}  }
  \leq \frac{   \d^3 \overline{\Pi}(\delta) }{   ( \delta^2 \overline{\Pi}(\delta) )^{\f{3}{2}}  } 
   = \frac{1}{\overline{\Pi}(\delta)^{\f{1}{2}}} \rightarrow 0. 
     \]

\end{proof}

\n Next we will prove the almost sure convergence result for $L(t,\delta)$. If there is a drift and the L\'evy measure is finite,  then the result is trivial. So we need only consider cases with infinite L\'evy measure, and begin with the zero drift case. Using a Borel-Cantelli argument (see \cite[p32]{k06} for details), we shall prove that   $ \liminf_{\d\to0}L(t,\delta)/\mu(\delta) = \limsup_{\d\to0} L(t,\delta)/\mu(\delta) = t$ almost surely.%  and limsup of $L(t,\delta)/\mu(\delta)$ as $\d\to0$ are both almost surely equal to $t$. 

 \n First, we will  prove the  almost sure convergence to $t$ along a  subsequence $\delta_n$ converging to zero. Then, by monotonicity  of $\mu(\d)$ and $L(t,\d)$, we will deduce that for all $\d$ between $\d_n$ and $\d_{n+1}$,  $L(t,\d)/\mu(\d)$ also tends to $t$ as $\d_n\to0$.

%Now we can establish an almost sure convergence result analogous to \cite[Theorem 1]{s14}, which follows a similar argument to \cite[Lemma 3.5]{s14}.

%*I could mention how convergence in probability follows quickly by Chebyshev's, but I don't think there's any need to.*
\newpage
\begin{proof}[Proof of Theorem \mbox{\ref{lslln}}]

 \n For all $\varepsilon>0$, by Chebyshev's inequality and Remark~\mbox{\ref{meanvar}}, \begin{equation*} \begin{split}    \sum_n \ \  \mathbb{P} \  \Big( \ \  \Big| \ \frac{L(t,\delta_n )}{t\mu(\delta_n ) } -1\ \Big|  >   \varepsilon \  \Big)   \  &\leq \hspace{2 mm} \frac{1}{\varepsilon^2} \sum_n   \frac{ \text{Var} \big( L(t,\delta_n )\big)}{t^2\mu(\delta_n )^2  } 
     \\ =     \frac{1}{\varepsilon^2} \sum_n   \frac{  \frac{ t}{\delta_n^2} \big( \int_0^{\delta_n}  x^2  \Pi(dx)   + \delta_n^2 \overline{\Pi}(\delta_n)  \big) }{ \frac{  t^2 }{   \delta_n^2 } \big(    \int_0^{\delta_n}  x  \Pi(dx)   + \delta_n \overline{\Pi}(\delta_n)    \big)^2     }   &=  \   \frac{1}{t \varepsilon^2} \sum_n   \frac{  \big( \int_0^{\delta_n}  x^2  \Pi(dx)   + \delta_n^2 \overline{\Pi}(\delta_n)  \big) }{  \big(    \int_0^{\delta_n}  x  \Pi(dx)   + \delta_n \overline{\Pi}(\delta_n)    \big)^2     }   
      \end{split} \end{equation*}  \b{equation}  \label{summ}    \leq    \frac{1}{t \varepsilon^2} \sum_n   \frac{  \delta_n \big( \int_0^{\delta_n}  x \Pi(dx)  +  \delta_n \overline{\Pi}(\delta_n)  \big) }{  \big(    \int_0^{\delta_n}  x  \Pi(dx)   + \delta_n \overline{\Pi}(\delta_n)    \big)^2     }     = \   \frac{1}{ t \varepsilon^2} \sum_n \frac{1}{\mu(\delta_n)} . \hspace{2.4cm} \e{equation}
      
\n Recall that  $\mu(\delta)  = \int_0^\infty  \f{1}{\d} ( x \wedge \delta) \ \Pi(dx) $, so since  $ \f{1}{\d} ( x \wedge \delta)$ is non-decreasing as $\delta$ decreases, it follows that $\mu(\delta)$ is non-decreasing as $\delta$ decreases. Now, $\lim_{\d\to0}\mu(\d) = \infty$, and $\mu$ is continuous, so it follows that for any fixed $r \in (0,1)$ there is a decreasing sequence $\delta_n$ such that \(\mu(\delta_n) = r^{-n}\) for each $n$. Then $\l( \ref{summ} \r)$ is finite,  so by the  Borel-Cantelli lemma,  $\lim_{n\to\infty} L(t,\d_n/\mu(\d_n) =t$ almost surely.

\n When there is no drift, $L(t,\delta)  $ is given by changing the original subordinator's jump sizes from $y$ to $ \f{1}{\d} (y\wedge \delta)$. By monotonicity of this map, it follows that for a fixed sample path of the original subordinator, each individual jump of the process $L(t, \delta_{n+1} )$  is at least as big as the corresponding jump of the process $L(t, \delta_{n} )$.   So  $L(t,\delta)$ is   non-decreasing as $\delta$ decreases, and so for all $\delta_{n+1} \leq \delta \leq \delta_n$,   \[    
 \frac{ L(t,\delta_{n} )  }{  t\mu(\delta_n)   }  \frac{  \mu(\delta_{n})  }{  \mu(\delta_{n+1} ) } \leq \frac{L(t,\delta )  }{t\mu(\delta)}  
  \leq  \frac{ L(t,\delta_{n+1} ) }{t \mu(\delta_n) }  
  = \frac{ L(t,\delta_{n+1} )  }{ t \mu(\delta_{n+1})   }   \frac{  \mu(\delta_{n+1} ) }{  \mu(\delta_{n})  }. 
 \] Then by our choice of the subsequence $\d_n$,  it follows that for all $\delta_{n+1} \leq \delta \leq \delta_n$, 
 \b{equation} \label{lll}
   r  \frac{ L(t,\delta_{n} )  }{  t\mu(\delta_n)   }  \leq   \frac{ L(t,\delta ) }{t \mu(\delta) }    \leq  \frac{1}{r}   \frac{ L(t,\delta_{n+1} )  }{t  \mu(\delta_{n+1})   }   ,
   \e{equation}  and since $\lim_{n\to\infty}   L(t,\d_n)/\mu(\d_n)=t$, it follows that \[rt \leq \liminf_{\d\to0}  \f{ L(t,\d)}{\mu(\d) } \leq\limsup_{\d\to0}    \f{L(t,\d)}{\mu(\d)}  \leq \f{t}{r}. \]
Taking limits as $r\to 1$,   it follows that $\lim_{\d\to0} L(t,\d)/\mu(\d) =t$ almost surely.

\n For a process with a  positive drift $\textup{d}>0$ and infinite L\'evy measure,  denote the scaling term obtained by removing the drift as $\hat{\mu}(\delta):=\mu(\delta) - \textup{d}/\delta$. Then the above Borel-Cantelli argument for $\hat{\mu}$ yields the almost sure limit along a subsequence $\hat{\delta}_n$ as in $\l(\ref{summ} \r)$. Then since the functions $\mu(\delta)$ and $L(t,\delta)$ are again monotone in $\d$ when there is a drift, the  argument applies as in $\l(\ref{lll} \r)$. %\vspace{-0.41cm}

% \qedhere
  \end{proof}

%\newpage

\begin{remark}  \textup{ Theorem \mbox{\ref{lslln}} is formulated in terms of the characteristics of the subordinator (i.e.\ the drift and L\'evy measure). For $N(t,\delta)$, the almost sure behaviour in Theorem \mbox{\ref{nslln}} is formulated in terms of the renewal function, and in order to write this in terms of the characteristics, the expression is  more complicated than for $L(t,\delta)$. For details, see \cite[Corollary 1]{s14} and \cite[Prop 1]{ds11}, the latter of which is very powerful for understanding the asymptotics of $U(\d)$ for subordinators with a positive drift, significantly improving upon results in \cite{cks11}. } \end{remark} %I don't think it is possible to easily express $L$'s a.s.\ asymptotic behaviour in terms of the renewal function. %It's not clear if this would be interesting, as the expression $\mu(\delta)$ is quite nice to work with anyway. 

 %\begin{remark}\textup{ In general, if two scaled sequences $A_n/\mu_n$, $B_n/\eta_n$ have a.s.\ limits, $A_n$ satisfies a CLT result, and the ratio betwee $\mathbb{E}[A_n]$ and $\text{Var}(A_n)$ is bounded, then we can extend to a CLT for  $B_n$. This is indeed the case for $L(t,\delta)$, as we show using  \ref{meanvar} and Jensen's inequality.             }\end{remark}

\section{Extensions and Special Cases} \label{ext}
%\subsection{Extensions: Box-Counting Dimension of the Graph}

%\section{Extensions: Box-Counting Dimension of the Graph of a Subordinator}
\subsection{Extensions: Box-Counting Dimension of the Graph}

The graph of a subordinator $X$ up to time $t$ is the set $\{   (s,X_s) :   0 \leq s \leq t   \}$. The box-counting dimensions of the range and graph are closely related. This is evident when we consider the mesh box counting schemes $M_G(t,\delta)$, $M_R(t,\delta)$, denoting graph and range respectively. The mesh box-counting scheme counts the number of boxes in a lattice of side length $\delta$ to intersect with a set.

\begin{remark}  \label{propn5.1} \textup{For every subordinator with infinite L\'evy measure or a positive drift,  \(M_G(t,\delta) = \l\lfloor  t / \d \r\rfloor   + M_R(t,\delta), \) where  $\lfloor \cdot \rfloor$ denotes the floor function.   Indeed, $M_R(t,\delta)$ increases by 1 if and only if $M_G(t,\delta)$ increases by 1 and the new box for the graph lies  directly above the previous box. For each integer $n$, $M_G(t,\delta)$  also increases at  time $n\delta$, the new box  directly to the right of the previous box.}\end{remark}
 \begin{remark} \textup{
 It follows that  the graph of every subordinator $X$ has the same box-counting dimension as the range of $X^\prime_t := t+X_t  $, the original process plus a unit drift. 
 }\end{remark}

%In particular, $L_g(\delta,t,X)  := \frac{t}{\delta}   + L_r(\delta,t,X) = L_r(\delta,t,X^\prime)$ defines the box-dimension of the graph, where $L_r$ is the previous definition $L$ for the range (and similar results follow for $M,N$).  So the counting scheme ``L'' differs by a deterministic time factor.  

\begin{proposition} \label{propn5.2} For every subordinator with drift \textup{d} $>0$, the box-counting dimensions of the range and graph agree almost surely.\end{proposition} 

%So the above theory applies for the range and for the graph, when there is drift. With the packing/Hausdorff bounds in mind, note that anything we prove here for the range also has a counterpart for the graph, so these results give insight into the packing/Hausdorff dimensions of both the range and graph of a subordinator with drift.

%This follows immediately from  the above relationship, if we accept that the dimension of the range is independent of the value of d, so long as   d $>0$. But that is proven in Corollary 3, and to avoid circular reasoning we  require a direct proof.
%\newpage

\begin{proof}[Proof of Proposition \mbox{\ref{propn5.2}}]
Letting $T_{(\d,\infty)}$ denote the first passage time of the subordinator above $\d$, consider an optimal covering of the graph with squares of side length $\d$ as follows: 

\n Starting with $[0,\delta] \times [0,\delta]$, at time $T_1 := \min(T_{(\delta,\infty)} , \delta)$, add a new box $[T_1 , T_1 + \delta] \times [X_{T_1} , X_{T_1} + \delta]$, and so on. Denote the number of these boxes  by $N_G(t,\delta)$, and write $N_R(t,\delta)$ as the optimal number of boxes needed to cover the range. 

\n If d $\geq1$, then we have $T_1 = T_{(\delta,\infty)}$ because $X_\delta \geq  \text{d}  \delta$. It follows that each time $N_{G}(t,\delta)$ increases by $1$, so does $N_{R}(t,\delta)$, and vice versa, so $N_{G}(t,\delta) = N_{R}(t,\delta)$, and  the box-counting dimension of the range and graph are equal when $\textup{d}\geq 1$.

\n For d $\in (0,1)$, a similar argument applies with a  covering of $\f{\d}{\textup{d}} \times \d$ rectangles rather than $\d \times \d$ squares. Starting with $[0, \f{\d}{\textup{d}} ] \times [0,\delta]$, at time $T_1$, add a new box $[T_1 , T_1 +\f{\d}{\textup{d}}] \times [X_{T_1} , X_{T_1} + \delta]$, and so on. The number of these boxes is again $N_{R}(\delta ,t)$, since $X_{\f{\d}{\textup{d}}} \geq \delta$.  By Remark \mbox{ \ref{freedom} }, this covering of rectangles can still be used to define the box-counting dimension of the range, since for $k := \l\lceil    \f{1}{\textup{d}} \r\rceil$, with $N_{G} (t,\delta)$ and $N_{G}^\prime (t,\delta)$ as  the number of squares and of rectangles respectively, \[ N_{G}^{\prime} (t,\delta) \leq N_{G} (t,\delta) \leq k \ N_{G}^{\prime}  (t, \delta / k ) .    \]

   \end{proof}

\begin{remark} \textup{ \label{bv} The box-counting dimension of the graph of every subordinator is 1 almost surely,  since subordinators have bounded variation (BV) almost surely. 
 The same is true for the graph of all BV functions/processes, including in particular every L\'evy process without a Gaussian component, whose L\'evy measure satisfies $\int (1 \wedge |x|)\Pi(dx)<\infty$. 
 By Proposition  \mbox{\ref{propn5.2}}, the  box-counting dimension of the range of every subordinator with drift  $\textup{d}>0$ is 1 almost surely.  }  \end{remark}

%\begin{remark} \textup{The graph of any L\'evy process $(X_t)_{t\geq0}$ can be interpreted as the range of the L\'evy process $(t,X_t)_{t\geq0}$.  }

%\end{remark}

%\begin{comment}
%\paragraph{Corollary} 

%\begin{proof}    With $N_g (\delta, t, X)$ denoting the optimal box-count for the process $X$ up to time $t$, and similarly for the others, first observe that      $     N_g (\delta, t, X) =      N_g(\delta,t, X^{(0,\delta)}) + Y_t^{\overline{\Pi}(\delta)}    $, where $Y_t^{\overline{\Pi}(\delta)}  $ is the counting Poisson process counting the ``large'' jumps, and $X^{(0,\delta)}$ is the process of small jumps.   Now, this compares to the grid-count:  $$       \asymp   \  \ M_g(\delta,t, X^{(0,\delta)}) + Y_t^{\overline{\Pi}(\delta)}  $$ 

%But because  $X^{(0,\delta)}$ has only small jumps, it cannot  ``skip'' any boxes. So the grid-count $M_g(\delta,t, X^{(0,\delta)})$ increases exactly when the process passes above $\delta, 2\delta, \dots$ , and also at each time $\delta, 2\delta, \dots$. So we can simply write  $$  M_g(\delta,t, X^{(0,\delta)}) + Y_t^{\overline{\Pi}(\delta)}   = \Big\lceil \frac{t}{\delta}  \Big\rceil    +  \Big\lceil \frac{X_t^{(0,\delta)}}{\delta}  \Big\rceil  + Y_t^{\overline{\Pi}(\delta)}  \asymp   L(\delta , t, X^{\prime}) , $$   where the asymptotic relation holds as the two sides are within $2$ of eachother.  \end{proof}

\newpage

\subsection{Special Cases: Regular Variation of the Laplace Exponent}   

%\paragraph{Corollary 1}   The upper/lower box-counting dimension has a nice expression in terms of the integrated tail, respectively  $$  \limsup_{\delta\rightarrow 0+}  \frac{ \log\big(  \frac{1}{\delta} I(\delta)\big)}{\log(\frac{1}{\delta})} , \quad  \liminf_{\delta\rightarrow 0+}  \frac{ \log\big(  \frac{1}{\delta} I(\delta)\big)}{\log(\frac{1}{\delta})}.$$ Note that a similar result arises from [Thm1, 7], but with $U(\delta)^{-1}$ in place of $\frac{1}{\delta} I(\delta)$. This is consistent with the relationship $ \frac{1}{\delta} I(\delta)  \asymp    U(\delta)^{-1}   $.

%Corollaries \ref{coroll1} and  \ref{coroll2} specialise respectively \ref{asymp} and \ref{lslln} to the case where the Laplace exponent is regularly varying with index $\alpha \in (0,1)$.   
%Corollary 3 expands on Section 3 to show that box-counting dimension of the graph is almost surely 1, and similarly for the range when d $>0$.

\n Corollary \mbox{\ref{coroll2}}  is analogous to \cite[Corollary 2]{s14}, with $L(t,\delta)$ in place of $N(t,\delta)$. This allows very fine comparisons, not visible at the log-scale, to be made  between  subordinators whose Laplace exponents  are regularly varying with the same index.  

\begin{corollary} \label{coroll2} Consider a subordinator whose Laplace exponent is regularly varying at infinity, such that $\Phi(\lambda) \sim  \lambda^\alpha  F(\lambda)$ for $\alpha \in(0,1)$, where $F(\cdot)$ is a slowly varying function. Then   almost surely as $\d\to0$, for all $t>0$, \[  L(t,\delta)  \sim     \frac{t  \delta^{-\alpha} F\l(\frac{1}{\delta}\r)  }{\Gamma( 2-\alpha ) }       .\]   \end{corollary}

\b{proof}[Proof of Corollary \mbox{\ref{coroll2}}]

  Note that $\textup{d}=0$, i.e.\ there is no drift, when the Laplace exponent is regularly varying of index $\alpha\in(0,1)$. By Theorem \mbox{\ref{lslln}}, as $\d\to0$,
   \[
   L(t,\d) \sim t \mu(\d) = \f{t I(\d)}{\d}  =   \f{t}{\d} \int_0^\d \overline{\Pi}(x)dx.
   \] Since $\Phi$ is regularly varying at $0$,   as $x\to0$, $\overline{\Pi}(x) \sim \Phi(\f{1}{x})/\Gamma(1-\alpha)$ (see \cite[p75]{b98}). Then by Karamata's Theorem  (see \cite[Prop. 1.5.8]{bgt89}),  almost surely as $\d\to0$,
   \[ 
    L(t,\d)   \sim    \f{ t \d^{-\alpha} F\l(  \f{1}{\d} \r)  }{ \Gamma(2-\alpha)   }.
   \]
 \e{proof}

\n  Corollary \mbox{\ref{coroll1}} strengthens the result of Theorem \mbox{\ref{asymp}} when the Laplace exponent $\Phi$ is regularly varying. The result can not be strengthened in general, as the relationship between $\mu(\delta)$ and $U(\delta)^{-1}$ is ``$\asymp$'' rather than ``$\sim$'' (see\mbox{\cite[Prop. 1.4]{b99}}).

%\paragraph{Corollary 1}
\begin{corollary} \label{coroll1} For a subordinator with Laplace exponent $\Phi$ regularly varying at infinity with index $\alpha \in (0,1)$, for all $t>0$, almost surely as $\d\to0$,  \[ N(t,\delta)  \sim    \Gamma(2-\alpha) \Gamma(1+\alpha)   L(t,\delta)   %=   \frac{\pi \alpha (1- \alpha )}{ \sin( \pi \alpha  )}  L(t,\delta)
.  \] \end{corollary} 

\n Corollary \mbox{\ref{coroll1}}  follows immediately from Corollary \mbox{\ref{coroll2}} and \cite[Corollary 2]{s14}, which says that when the Laplace exponent $\Phi$ is regularly varying at infinity, such that $\Phi(\lambda) \sim  \lambda^\alpha  F(\lambda)$ for $\alpha \in(0,1)$, where $F(\cdot)$ is a slowly varying function, for all $t>0$, almost surely as $\d\to0$, \[N(t,\d)\sim    \Gamma(1+\alpha) t \d^{-\alpha} F\l( \f{ 1}{\d} \r) .   \]
%\begin{proof}[Proof of Corollary \mbox{\ref{coroll1}}]  From\mbox{\cite[Proposition 1.5]{b99}}, $ U(\delta) \sim \frac{1}{ \Gamma(1+\alpha) \Phi(\frac{1}{\delta})   } $ and $  \overline{\Pi}(\delta) \sim  \frac{\Phi(\frac{1}{\delta}) }{ \Gamma(1-\alpha)   }$.

%\n Moreover, from\mbox{\cite[p73]{b98}},  $\Pi(dx) =  \frac{ \alpha}{\Gamma(1-\alpha)} x^{-1-\alpha}  dx  $, so it follows that \[\mu(\delta) =  \frac{1}{\delta}\int_0^\delta    x \Pi(dx)    +  \overline{\Pi}(\delta) =   \frac{\alpha}{\delta \Gamma(1-\alpha)}  \frac{\delta^{1-\alpha}}{1-\alpha}  + \frac{\delta^{-\alpha}}{\Gamma(1-\alpha)}\] \[=\frac{\delta^{-\alpha}}{ (1-\alpha) \Gamma(1-\alpha)} \sim  \frac{\Phi(\frac{1}{\delta})}{\Gamma(2-\alpha)} .\] 

%\n Then it follows that almost surely as $\d\to0$, \[ L(t,\delta) \sim t \mu(\delta) \sim   \frac{t \Phi(\frac{1}{\delta})}{\Gamma(2-\alpha)}  ,   \quad  N(t,\delta)  \sim  \frac{t}{U(\delta)} \sim t \Gamma(1+\alpha) \Phi\Big(\frac{1}{\delta}\Big), 
%\] 
%and then we can conclude that almost surely as $\d\to0$, 
  %\[    N(t,\delta)  \sim  \Gamma(2-\alpha) \Gamma(1+\alpha)   L(t,\delta)   =    \frac{\pi \alpha (1- \alpha )}{ \sin( \pi \alpha  )}  L(t,\delta) .  \] 

   %\end{proof}

%It is interesting to note that here the asymptotics of $\mu(\delta)$ depend on both the tail \textbf{and} the integral near zero. In general I think one is likely to dominate the other asymptotically.
\vspace{0.1cm}
\vspace{0.1cm}
\begin{remark} For $\alpha \in (0,1)$,   $ \Gamma(2-\alpha) \Gamma(1+\alpha) $  takes values between $\pi /4$ and  $1$. So $L(t,\delta)$ and $N(t,\delta)$  are closely related when the Laplace exponent is regularly varying, but as $\d\to0$, $L(t,\delta)$ grows to infinity slightly faster than $N(t,\delta)$.  \end{remark}

\vspace{0.1cm}

\vspace{0.1cm}

\vspace{0.1cm}
\vspace{0.1cm}
\vspace{0.1cm}
\subsection{Acknowledgements} Many thanks to Mladen Savov for guiding the author towards this interesting topic, and for numerous helpful discussions related to  this work. Further thanks to Ron Doney for his feedback on an early draft of this paper, and thanks to an anonymous referee for their valuable comments on this work.

%\newpage
  %Acknowledge Mladen's help, both with carrying out the work, and with directing me towards an interesting research question. Also acknowledge Ron if we talk about the paper and he offers any advice.

%Read many papers' acknowledgements sections, from various journals/authors/fields, in order to get a feel for what to say. The same applies for things like section titles, ordering of theorems (do I state them all at the start, or just plow through?), and everything else.

%Have a search and read around for more papers to reference. I'm sure there are several papers worth mentioning on fractal study of levy processes.
 
%\nocite{*}
 
\bibliography{boxdim.bib}
\bibliographystyle{plain}

\end{document}